\documentclass[11pt,oneside,a4paper]{amsart}
\usepackage{amsmath}
\usepackage{amsthm}
\usepackage{amssymb}
\usepackage{xcolor}
\usepackage{enumerate}

\usepackage{mathtools}
\usepackage{hyperref}
\usepackage[numbers]{natbib}
\hypersetup{
	breaklinks=true,
	colorlinks=true,
	linkcolor=teal,
	citecolor=violet}

\theoremstyle{plain}
\newtheorem{theorem}{Theorem}[section]
\newtheorem{corollary}[theorem]{Corollary}
\newtheorem{proposition}[theorem]{Proposition}

\theoremstyle{definition}

\newtheorem{remark}[theorem]{Remark}
\newtheorem{remarks}[theorem]{Remarks}

\numberwithin{theorem}{section}
\numberwithin{equation}{section}

\renewcommand\d{\mathrm{d}}
\newcommand\dd{\,\d}
\newcommand{\ddov}[1]{\, \frac{\d{#1}}{{#1}}}

\newcommand{\N}{\mathbb{N}}
\newcommand{\R}{\mathbb{R}}
\newcommand{\rn}{\R^n}
\usepackage{mathrsfs}
\newcommand{\M}{\mathscr{M}}

\DeclareMathOperator{\spt}{supp}

\title{Maximal noncompactness of limiting Sobolev embeddings}

\author{Jan Lang, Zden\v ek Mihula and Lubo\v s Pick}

\address{Jan Lang, Department of Mathematics, The Ohio State University, 231 West 18th Avenue, Columbus, OH 43210-1174; Czech Technical University in Prague, Faculty of Electrical Engineering, Department of Mathematics, Technick\'a~2, 166~27 Praha~6, Czech Republic}
\email{lang.162@osu.edu}
\urladdr{\href{https://orcid.org/0000-0003-1582-7273}{0000-0003-1582-7273}}

\address{Zden\v ek Mihula, Czech Technical University in Prague, Faculty of Electrical Engineering, Department of Mathematics, Technick\'a~2, 166~27 Praha~6, Czech Republic}
\email{mihulzde@fel.cvut.cz}
\urladdr{\href{https://orcid.org/0000-0001-6962-7635}{0000-0001-6962-7635}}

\address{Lubo\v{s} Pick, Department of Mathematical Analysis, Faculty of Mathematics and Physics, Charles University, Sokolovsk\'a~83, 186~75 Praha~8, Czech Republic}
\email{pick@karlin.mff.cuni.cz}
\urladdr{\href{https://orcid.org/0000-0002-3584-1454}{0000-0002-3584-1454}}

\begin{document}



\thanks{This research was supported in part by grant No.~23-04720S of the Czech Science Foundation.}

\begin{abstract}
    We develop a new method suitable for establishing lower bounds on the ball measure of noncompactness of operators acting between considerably general quasinormed function spaces. This new method removes some of the restrictions oft-presented in the previous work. Most notably, the target function space need not be disjointly superadditive nor equipped with a norm. Instead, a property that is far more often at our disposal is exploited\textemdash namely the absolute continuity of the target quasinorm.
    
		We use this new method to prove that limiting Sobolev embeddings into spaces of Brezis--Wainger type are so-called maximally noncompact, i.e., their ball measure of noncompactness is the worst possible.
\end{abstract}

\maketitle

\section{Introduction}

Maximal noncompactness is a useful modern function-analytic tool whose primary purpose is to provide refined information about noncompact mappings. Although compactness is a topological property, the definition of the maximal noncompactness is purely geometric. For example, it is not invariant with respect to topologically equivalent (quasi)norms. The concept is intimately connected with that of the ball measure of noncompactness, whose earliest occurrences reach as far as the 1930s\textemdash notably thanks to the advances by Kuratovskii (see~\cite{Kur:30})\textemdash and which has been applied to various problems in analysis and its applications (see~\cite{BG:80,Dar:55,EE:18,Sad:68}). It is one of the successful tools for compensating for the lack of compactness; another that has been of interest is for example the concentrated compactness principle (see~e.g.~\cite{Bre:83,Lio1:85,Lio2:85,Str:84}).

It is quite easy to realize that the ball measure of noncompactness, $\beta(T)$, of an operator $T$ acting between (quasi)normed linear spaces $X$ and $Y$, defined as the infimum of radii $\varrho>0$ for which $T(B_X)$ can be covered by finitely many balls in $Y$ with radius $\varrho$, cannot exceed the operator norm, $\|T\|$, of $T$. At the opposite end, the lowest value it can attain is zero, and it is an easy exercise to see that $\beta(T)=0$ is a necessary and sufficient condition for $T$ to be compact. Hence, $\beta(T)=\|T\|$ indicates, from the point of view of finite coverings, something like `the worst possible degree of noncompactness' of the operator $T$.

The notion of the maximal noncompactness itself (i.e., $\beta(T)=\|T\|$) is, however, far more recent. The notion was first used in~\cite{Lan:21}, and it surfaced in connection with sharp (nonlimiting/subcritical) Sobolev embeddings, which are typically noncompact (e.g, \cite{Cav:19,Cav:22,Ker:08,Sla:15}), and for which the additional information about how bad their noncompactness is, can be of use. Results on the maximal noncompactness of various operators had been known before, but they had not been called that way (cf.~\cite{EE:18} or~\cite{Hen:03}).

Noncompact Sobolev embeddings are often maximally noncompact. The classical nonlimiting embedding of a first-order homogeneous Sobolev space into the smallest possible Lebesgue space, namely
\begin{equation}\label{E:sobolev-lebesgue-nonlimiting}
    V^{1,p}_0(\Omega)\to L^{\frac{np}{n-p}}(\Omega),
\end{equation}
where $n\in\N$, $n\ge 2$, $1\le p<n$, and $\Omega\subseteq\mathbb R^n$ is a bounded domain (cf.~e.g.~\cite{Ada:75,Maz:11}), was shown to be maximally noncompact in~\cite{Hen:03}. The precise definition of first-order homogeneous Sobolev spaces is given in Section~\ref{sec:prel}. When one ventures out of the realm of Lebesgue spaces, enhancements of the embedding~\eqref{E:sobolev-lebesgue-nonlimiting} are available. Most notably, with the help of two-parameter Lorentz spaces, one has  
\begin{equation}\label{E:sobolev-lorentz-nonlimiting}
    V^{1,p}_0(\Omega)\to L^{\frac{np}{n-p},p}(\Omega),
\end{equation}
under the same restrictions on the parameters as in~\eqref{E:sobolev-lebesgue-nonlimiting}.

Now, while it is easy to verify that~\eqref{E:sobolev-lorentz-nonlimiting} is sharper than~\eqref{E:sobolev-lebesgue-nonlimiting}, owing to the (proper) inclusion $L^{\frac{np}{n-p},p}(\Omega)\subsetneq L^{\frac{np}{n-p}}(\Omega)$, there is no obvious way of exploiting the information that one embedding is maximally noncompact for proving that so is the other. In fact, this seeming discrepancy is universal, not limited only to these two particular embeddings. So, the maximal noncompactness of~\eqref{E:sobolev-lorentz-nonlimiting} needed separate treatment, and indeed received it in~\cite{Bou:19}. In fact, it was shown there that any of the Sobolev embeddings of the form
\begin{equation}\label{E:sobolev-lorentz-lorentz}
    V^{1}_0L^{p,q}(\Omega)\to L^{\frac{np}{n-p},r}(\Omega),
\end{equation}
where $1\leq q \leq r < \infty$, is still maximally noncompact. Here $V^{1}_0L^{p,q}(\Omega)$ is a first-order homogeneous Sobolev space built upon the Lorentz space $L^{p,q}(\Omega)$ (see Section~\ref{sec:prel} for the precise definition). This result left open the case when $r=\infty$, where the method used in both \cite{Bou:19} and~\cite{Hen:03} fails. The method applied there was based on a combination of the disjoint superadditivity of the target space and a certain shrinking property of the embedding. Recall that a functional $\|\cdot\|_Y$ is said to be $\gamma$-\emph{disjointly superadditive} for some $\gamma\in(0, \infty)$ if there is a constant $C > 0$ such that, for every $k\in\N$ and functions $\{f_j\}_{j=1}^{k}\subseteq Y$ having pairwise disjoint supports, one has
    \begin{equation*}
        \sum_{j=1}^{k}\|f_j\|_{Y}^{\gamma} \le C\left\|\sum_{j=1}^{k}f_j\right\|_{Y}^{\gamma}.
    \end{equation*}
    In the theory of Banach spaces/lattices, this property is often called a lower $\gamma$-estimate, see~\cite{LT:79} for example. Remarkably, it was proved later in \cite{Lan:21} by completely different methods that the embedding 
\begin{equation}\label{E:sobolev-weak-lebesgue}
    V^{1}_0 L^{p,q}(\Omega)\to L^{\frac{np}{n-p},\infty}(\Omega),
\end{equation}
where $q\in[1,\infty]$, retains its maximal noncompactness despite the fact that the weak Lebesgue space $L^{\frac{np}{n-p},\infty}(\Omega)$ is not disjointly superadditive.

The \emph{limiting} Sobolev embedding (corresponding to taking $p=n$ in the domain space) is much more difficult (and therefore much more interesting) to handle. Arguably the most notorious version of the limiting Sobolev embedding uses an exponential-type Orlicz space for its target, and reads as
\begin{equation}\label{E:trudinger}
    V^{1,n}_0(\Omega)\to \exp L^{\frac{n}{n-1}}(\Omega),
\end{equation}
see~\cite{Poh:65,Str:71,Tru:67,Yud:61}. This embedding is known to be sharp as far as Orlicz target spaces are concerned, and, rather unsurprisingly, also to be noncompact (see~\cite{Hem:70}). Results on the maximal noncompactness of~\eqref{E:trudinger} are available, too. It was shown in~\cite{Hen:03}, using extremal properties of certain radially decreasing functions from~\cite{Mos:70}, that~\eqref{E:trudinger} is maximally noncompact as long as the measure of $\Omega$ is small enough. 

A Lorentz-like refinement of~\eqref{E:trudinger} is possible, too, but the two-parameter scale of Lorentz spaces is not sufficient for it. The result can be stated for example in the form
\begin{equation}\label{E:sobolev-lorentz-zygmund}
    V^{1,n}_0(\Omega)\to L^{\infty,q,\frac{1}{n}-\frac{1}{q}-1}(\Omega),
\end{equation}
where $q\in[n,\infty)$. The target space $L^{\infty,q,\frac{1}{n}-\frac{1}{q}-1}(\Omega)$ is an instance of the so-called Lorentz--Zyg\-mund spaces. Lorentz--Zygmund spaces were introduced in~\cite{Ben:80} and further treated, e.g., in~\cite{OP:99}. For various statements, proofs, and further details concerning the embedding~\eqref{E:sobolev-lorentz-zygmund}, see e.g.~\cite{Bre:80,Bru:79,Cwi:98,Han:79,Mal:02,Maz:11}. The embedding \eqref{E:sobolev-lorentz-zygmund} indeed improves \eqref{E:trudinger} because $L^{\infty,q_1,\frac{1}{n}-\frac{1}{q_1}-1}(\Omega)\subsetneq L^{\infty,q_2,\frac{1}{n}-\frac{1}{q_2}-1}(\Omega)\subsetneq \exp L^{\frac{n}{n-1}}(\Omega)$ for every $n\leq q_1 < q_2 < \infty$. The weak variant of \eqref{E:sobolev-lorentz-zygmund}, corresponding to $q = \infty$, reads as
\begin{equation}\label{E:sobolev-lorentz-zygmund_weak}
    V^{1,n}_0(\Omega)\to L^{\infty,\infty,\frac{1}{n}-1}(\Omega).
\end{equation}

A common feature of the target spaces in~\eqref{E:sobolev-weak-lebesgue} and~\eqref{E:sobolev-lorentz-zygmund_weak} is that both are examples, albeit of quite a different nature, of the so-called Marcinkiewicz spaces, which play an important role for example in interpolation theory, or in the theory of rearrangement-invariant function spaces. Although the defining functionals for the spaces $\exp L^{\frac{n}{n-1}}(\Omega)$ and $L^{\infty,\infty,\frac{1}{n}-1}(\Omega)$ are equivalent, their mere equivalence is not enough for deriving the maximal noncompactness of \eqref{E:sobolev-lorentz-zygmund_weak} from that of \eqref{E:trudinger}.

The question of whether the refined limiting embeddings \eqref{E:sobolev-lorentz-zygmund} as well as their weak variant \eqref{E:sobolev-lorentz-zygmund_weak} are maximally noncompact has been open. We fill this gap in this paper. This question is interesting for several reasons. Notably, the target space in \eqref{E:sobolev-lorentz-zygmund} is neither a Marcinkiewicz space nor it is disjointly superadditive (the latter observation, although it is most likely not available in an explicit form, is hidden in~\cite{KM:04}, as we shall point out), and so the techniques that successfully worked in earlier approaches cannot be used here. As for \eqref{E:sobolev-lorentz-zygmund_weak}, although the target space is a Marcinkiewicz space as in \cite{Lan:22}, the method used there is not suitable for proving the maximal noncompactness of \eqref{E:sobolev-lorentz-zygmund_weak}. The reason is that the logarithmic function defining the Marcinkiewicz space $L^{\infty,\infty,\frac{1}{n}-1}(\Omega)$ grows too slowly (or rather it is too slowly varying). Nevertheless, the ball measure of noncompactness of embeddings into Marcinkiewicz spaces and their maximal noncompactness was treated in~\cite{MMMP:preprint}. Among other things, it was proved there that Marcinkiewicz spaces are more or less never disjointly superadditive (with a single notable exception when the space collapses to $L^1(\Omega)$, which is a theoretical possibility).

We would like to point out that there are some other technical difficulties that one has to overcome when dealing with \eqref{E:sobolev-lorentz-zygmund} and \eqref{E:sobolev-lorentz-zygmund_weak}, which are perhaps not entirely obvious. Notably, the functionals governing the target spaces are, strictly speaking, not norms, but merely quasinorms. Given the geometric nature of the maximal noncompactness, this unsurprisingly causes trouble. Although the functionals are equivalent to norms, this equivalence is not enough when dealing with properties of such a geometric nature such as the maximal noncompactness of a mapping. 

In this paper, we adopt a lateral point of view. Departing from the observation that Lorentz--Zygmund spaces are special instances of the classical Lorentz spaces of type Lambda, we first study the rather general question of maximal noncompactness of operators whose target space is one of these spaces. Our universal results then enable us to solve the open problem in the affirmative, more precisely, we will show that the embedding~\eqref{E:sobolev-lorentz-zygmund} is maximally noncompact, despite the target space not being disjointly superadditive nor equipped with a norm. Lacking disjoint superadditivity and not having the Marcinkiewicz structure at hand, we make use of other properties possessed by this structure, the key one being the absolute continuity of the target quasinorm. We will also show the maximal noncompactness of the embedding \eqref{E:sobolev-lorentz-zygmund_weak}. To this end, we exploit the new technique developed in \cite{MMMP:preprint}.

Our approach is based on a carefully tailored minimal axiomatization of the properties of functionals that are needed in order to obtain, as the ultimate goal, a suitable lower bound for the measure of noncompactness. This allows us to treat the problem in quite a general setting and to circumvent some of the obstacles. Let us note that, in particular, we do not assume the Fatou property, and we allow for quasinorms. This broadens the field of applications significantly. 

The axiomatization process leads us to introducing quasi-K\"othe function spaces. We further introduce a new geometric property of functionals, which we call uniform separation. It is worth noticing that while all $\alpha$-norms have it, not all quasinorms do, and, of course, the Aoki--Rolewicz theorem does not help for the same reasons as those stated above.


\section{Preliminaries}\label{sec:prel}
Throughout the paper, we assume that $(R,\mu)$ is a nonatomic measure space. We denote by $\M(R,\mu)$ the set of all $\mu$-measurable functions on $R$, and by $\M_0(R,\mu)$ those functions from $\M(R,\mu)$ that are finite $\mu$-a.e.

Let $Y$ be a quasi-Banach space contained in $\M_0(R,\mu)$ with a quasinorm $\|\cdot\|_Y$.

We say that the quasinorm of $Y$ is \emph{absolutely continuous} if every $f\in Y$ has absolutely continuous quasinorm\textemdash that is,
\begin{equation*}
\lim_{n\to \infty} \|f\chi_{E_n}\|_{Y} = 0
\end{equation*}
for every sequence $\{E_n\}_{n=1}^\infty\subseteq R$ of $\mu$-measurable sets such that $E_n\to\emptyset$ as $n\to\infty$.

We say that $Y$ is a \emph{quasi-Banach lattice} if for every $f\in \M_0(R,\mu)$ and $g\in Y$ such that $|f|\leq |g|$ $\mu$-a.e.~in $R$, we have $f\in Y$ and $\|f\|_Y \leq \|g\|_Y$.

We say that a quasi-Banach lattice $Y$ is a \emph{quasi-K\"othe function space} (cf.~\cite{CNS:03}) if for every $f\in Y$ and $E\subseteq R$ with $\mu(E) < \infty$, we have $\chi_E \in Y$ and $f\chi_E \in L^1(R, \mu)$.

The \emph{distributional function} of a function $f\in\M(R,\mu)$ is the function $f_*\colon(0, \infty) \to [0, \infty]$ defined as
\begin{equation*}
f_*(\lambda) = \mu(\{x\in R: |f(x)| > \lambda\}),\ \lambda\in(0, \infty).
\end{equation*}
The \emph{nonincreasing rearrangement} of $f\in\M(R,\mu)$ is the function $f^*\colon(0, \infty) \to [0, \infty]$ defined as
\begin{equation*}
f^*(t) = \inf\{\lambda > 0: f_*(\lambda)\leq t\},\ t\in(0, \infty).
\end{equation*}
The nonincreasing rearrangement satisfies
\begin{equation}\label{prel:star_of_sum_pointwise_est}
(f + g)^*(s + t) \leq f^*(s) + g^*(t) \quad \text{for every $s,t\in(0, \infty)$}
\end{equation}
and every $f,g\in\M_0(R,\mu)$.

Let $w\in\M_0(0, \mu(R))$ be a.e.~positive and $q\in(0, \infty)$. The \emph{Lambda space} $\Lambda^q_w(R,\mu)$ is defined as
\begin{equation*}
\Lambda^q_w(R,\mu) = \{f\in \M(R,\mu): \|f\|_{\Lambda^q_w(R,\mu)} < \infty\},
\end{equation*}
where
\begin{equation*}
\|f\|_{\Lambda^q_w(R,\mu)} = \left( \int_0^{\mu(R)} f^*(t)^q w(t) \dd t\right)^\frac1{q},\ f\in \M(R, \mu).
\end{equation*}
The functional $\|f\|_{\Lambda^q_w(R,\mu)}$ can be expressed as
\begin{equation}\label{prel:lambda_by_means_of_distributional_func}
\|f\|_{\Lambda^q_w(R,\mu)} = \left(q \int_0^\infty W(f_*(t)) t^{q - 1} \dd t\right)^\frac1{q}.
\end{equation}
Note that we have $\|\cdot\|_{\Lambda^q_w(R,\mu)} = \|\cdot\|_{L^q(R,\mu)}$ for $w\equiv 1$ (e.g., \cite{LL:01}).

From now on, we assume that
\begin{equation}\label{prel:Lambda_space_nontrivial}
0 < W(t) = \int_0^t w(s) \dd s < \infty \quad \text {for every $t\in(0,\mu(R))$},
\end{equation}
for otherwise $\Lambda^q_w(R,\mu)$ contains only the zero function.

Despite the terminology, $\Lambda^q_w(R,\mu)$ need not be a linear set (see \cite{CKMP:04}). The Lambda space $\Lambda^q_w(R,\mu)$ is a quasi-Banach space if and only if (see \cite{CS:93})
\begin{equation}\label{prel:Lambda_space_deltaTwo}
\sup_{t\in(0, \mu(R)/2)}\frac{W(2t)}{W(t)} < \infty.
\end{equation}
It is easy to see (owing to the dominated convergence theorem) that $\Lambda^q_w(R,\mu)$ is a quasi-Banach lattice whose quasinorm is absolutely continuous if (and only if) $\Lambda^q_w(R,\mu)$ is a quasi-Banach space. If, in addition,
\begin{equation}\label{prel:Lambda_space_quasiKothe}
\begin{cases}
\sup_{t\in(0, a)} \frac{t^q}{W(t)}  < \infty \quad &\text{if $q \in(0, 1]$}, \\
\int_0^{a} \big( \frac{t}{W(t)} \big)^\frac1{q-1} \dd t < \infty \quad &\text{if $q\in(1, \infty)$},
\end{cases}
\end{equation}
for every $a\in(0, \mu(R))$, then $\Lambda^q_w(R,\mu)$ is a quasi-K\"othe function space (see~\cite{S:93}).

Let $\mu(R) < \infty$ and $w_{p,q,\alpha}$ be defined as
\begin{equation}\label{prel:LZ_weight_def}
w_{p,q,\alpha}(t) = t^{\frac1{p} - \frac1{q}}\log\Big( \frac{2\mu(R)}{t} \Big)^{\alpha},\ t\in(0, \mu(R)),
\end{equation}
for $p,q\in(0, \infty]$ and $\alpha\in\R$. The \emph{Lorentz--Zygmund spaces} $L^{p,q,\alpha}(R,\mu)$ are defined as
\begin{equation*}
L^{p,q,\alpha}(R,\mu) = \{f\in\M(R, \mu): \|f\|_{L^{p,q,\alpha}(R,\mu)} = \|f^*w_{p,q,\alpha}\|_{L^q(0, \mu(R))} < \infty\}.
\end{equation*}
Note that $L^{p,q,\alpha}(R,\mu) = \Lambda^q_{w_{p,q,\alpha}^q}(R,\mu)$ for $q\in(0, \infty)$. Furthermore, note that $\|\cdot\|_{L^{p,p,0}(R, \mu)} = \|\cdot\|_{L^p(R, \mu)}$ for every $p\in(0, \infty]$. The spaces $L^{p,q}(R, \mu) = L^{p,q, 0}(R, \mu)$ are often called (two-parameter) \emph{Lorentz spaces}. 

Straightforward computations show that the Lorentz--Zygmund space $L^{p,q,\alpha}(R,\mu)$ is a quasi-K\"othe function space provided that one of the following conditions is satisfied:
\begin{equation}\label{prel:LZ_quasi_Kothe_cond}
\begin{cases}
&\text{$p = 1$, $q \in(0, 1]$, and $\alpha \geq 0$};\\
&\text{$p = 1$, $q \in (1, \infty]$, and $\alpha + \frac1{q} > 1$};\\
&\text{$p\in(1, \infty)$ and $q \in(0, \infty]$};\\
&\text{$p = \infty$, $q \in(0, \infty)$, and $\alpha + \frac1{q} < 0$};\\
&\text{$p = q = \infty$ and $\alpha \leq 0$}.
\end{cases}
\end{equation}
Furthermore, if $p=q=1$ and $\alpha\geq 0$ or $p,q\in(1, \infty)$ or $p = \infty$, $q\in[1, \infty)$, and $\alpha + \frac1{q} < 0$, then $\|\cdot\|_{L^{p,q,\alpha}}$ is equivalent to a norm and the Lorentz--Zygmund space $L^{p,q,\alpha}(R,\mu)$ is equivalent to a rearrangement-invariant Banach function space in the sense of \cite{BS_book:88} (see~\cite{OP:99}).

Given a domain $\Omega\subseteq\rn$ of finite measure and $p\in[1,\infty)$, the first-order homogeneous Sobolev space $V^{1,p}_0(\Omega)$ is the Banach space of all weakly differentiable functions $u$ on $\Omega$, endowed with the norm $\|u\|_{V^{1,p}_0(\Omega)} = \|\,|\nabla u|\,\|_{L^p(\Omega)}$, whose continuation by~$0$ outside~$\Omega$ is weakly differentiable and whose gradient belongs to $L^p(\Omega)$. The first-order homogeneous Sobolev space $V^{1}_0L^{p,q}(\Omega)$ built upon a Lorentz space $L^{p,q}(\Omega)$ is defined in the same way but with $L^p(\Omega)$ replaced by $L^{p,q}(\Omega)$.


\section{General theorem on maximal noncompactness}
We say that a quasinorm $\|\cdot\|_Y$ on a quasi-Banach space $Y$ is \emph{uniformly separating} if
for every $0 < r < R$ there is $\varepsilon_{r,R} > 0$ such that
\begin{equation}\label{thm:maximal_noncomp_ac_target:property_of_Y}
		\|f + g\|_Y \geq \varepsilon_{r,R}
\end{equation}
for every $f,g\in Y$ satisfying $\|f\|_Y \geq R > r \geq \|g\|_Y$.

Note that $\|\cdot\|_Y$ is uniformly separating if it is an $\alpha$-norm for some $\alpha\in(0, 1]$. We say that $\|\cdot\|_{Y}$ is an~\emph{$\alpha$-norm} if
\begin{equation*}
    \|f+g\|_Y^{\alpha}\le \|f\|_Y^{\alpha}+\|g\|_Y^{\alpha}
    \quad\text{for every $f,g\in Y$.}
\end{equation*}
Obviously, a $1$-norm is a norm. When $\|\cdot\|_Y$ is an $\alpha$-norm, then
\begin{equation*}
\|f + g\|_Y \geq \big( \|f\|_Y^\alpha - \|g\|_Y^\alpha \big)^{\frac1{\alpha}} \geq \varepsilon_{r, R} = \big( R^\alpha - r^\alpha \big)^\frac1{\alpha} > 0
\end{equation*}
for every $f,g\in Y$ such that $\|f\|_Y \geq R > r \geq \|g\|_Y$. A pivotal example of an $\alpha$-norm that is not a norm is the $L^p$-quasinorm for $p\in(0,1)$. A less obvious one is the Lorentz $L^{1,q}$-quasinorm for $q\in(0,1)$, which is a $q$-norm. Although not every quasinorm is an $\alpha$-norm for some $\alpha$, every quasinorm is equivalent to some $\alpha$-norm by virtue of the Aoki--Rolewicz theorem (e.g., \cite[Proposition~H.2]{BL:00}). However, even though every quasinorm is equivalent to an $\alpha$-norm, not every quasinorm is uniformly separating. To this end, let $Y = \R^2$ and
\begin{equation*}
\|(x,y)\|_Y = \begin{cases}
 2|x| \quad &\text{if $y = 0$}, \\
|x| + |y| \quad &\text{if $y\neq 0$}.
\end{cases}
\end{equation*}
Now, set $r = 3/2$ and $R = 2$. Let $\delta \in (0, 1/2)$, and consider $f = (1, 0)$ and $g = (-1, \delta)$. Then $\|f\|_Y = R > r \geq \| g \|_Y$, but $\|f + g\|_Y = \|(0, \delta)\|_Y = \delta$.

Having defined what uniformly separating quasinorms are, we are in a position to state and prove our main general theorem suitable for establishing lower bounds on the ball measure of noncompactness. Recall that the \emph{ball measure of noncompactness} of a bounded positively homogeneous operator $T\colon X \to Y$, where $X$ and $Y$ are quasi-Banach spaces, is defined as
\begin{align*}
    \beta(T\colon X \to Y) = \inf\Big\{ \varrho > 0:&\ \text{there are $m\in\N$ and $\{y_j\}_{j = 1}^m\subseteq Y$ such that} \\
    &\ T(B_X) \subseteq \bigcup_{j = 1}^m (y_j + \varrho B_X) \Big\}.
\end{align*}
Here (and below), $B_X$ is the closed unit ball of $X$. An operator  $T\colon X \to Y$ is said to be  \emph{positively homogeneous} if $\|T(\alpha x)\|_Y = \alpha \|Tx\|_Y$ for every $x\in X$ and every scalar $\alpha > 0$.
\begin{theorem}\label{thm:maximal_noncomp_ac_target}
Let $T$ be a bounded positively homogeneous operator from a quasi-Banach space $X$ to a quasi-K\"othe function space $Y\subseteq \M_0(R,\mu)$, where $(R,\mu)$ is a $\sigma$-finite measure space. Assume that $Y$ has absolutely continuous and uniformly separating quasinorm.

Let $\lambda > 0$. If there is a sequence $\{x_j\}_{j = 1}^\infty \subseteq B_X$ such that
\begin{equation}\label{thm:maximal_noncomp_ac_target:sequence_spt}
\spt Tx_j \rightarrow \emptyset \quad \text{as $j\to\infty$}
\end{equation}
and
\begin{equation}\label{thm:maximal_noncomp_ac_target:sequence_qnorm}
\|Tx_j\|_Y \geq \lambda \quad \text{for every $j\in\N$},
\end{equation}
then $\beta(T\colon X \to Y)\geq \lambda$.

In particular, if such a sequence $\{x_j\}_{j = 1}^\infty$ exists for every $\lambda\in(0, \|T\|_{X\to Y})$, then the operator $T$ is maximally noncompact.
\end{theorem}
\begin{proof}
Set $\beta = \beta(T\colon X \to Y)$. Suppose that $\beta < \lambda$.
Let $\{x_j\}_{j = 1}^\infty \subseteq B_X$ be a sequence satisfying \eqref{thm:maximal_noncomp_ac_target:sequence_spt} and \eqref{thm:maximal_noncomp_ac_target:sequence_qnorm}. Fix any $r \in (\beta, \lambda)$.
Since $r > \beta$, it follows from the definition of the measure of noncompactness that there are $m\in\N$ and functions $\{g_k\}_{k = 1}^m \subseteq Y$ such that
\begin{equation}\label{thm:maximal_noncomp_ac_target:finite_cover}
T(B_X) \subseteq \bigcup_{k = 1}^m \big( g_k + r B_Y \big).
\end{equation}
Thanks to \eqref{thm:maximal_noncomp_ac_target:finite_cover}, for every $j\in\N$ there is $k_j\in\{1, \dots, m\}$ such that
\begin{equation}\label{thm:maximal_noncomp_ac_target:eq1}
\|Tx_j - g_{k_j}\|_Y \leq r.
\end{equation}
Set
\begin{equation*}
h_j = g_{k_j}\chi_{\spt Tx_j} \quad \text{for every $j\in\N$}.
\end{equation*}
Clearly
\begin{equation*}
|h_j - Tx_j| \leq |g_{k_j} - Tx_j| \quad \text{$\mu$-a.e.~in $R$ for every $j\in\N$}.
\end{equation*}
Owing to this and the monotonicity of $\|\cdot\|_Y$, it follows from \eqref{thm:maximal_noncomp_ac_target:eq1} that
\begin{equation}\label{thm:maximal_noncomp_ac_target:eq2}
\|h_j - Tx_j\|_Y \leq r \quad \text{for every $j\in\N$}.
\end{equation}
Now, on the one hand, we can write
\begin{equation*}
    \|h_j\|_{Y}=\|Tx_j+h_j-Tx_j\|_{Y},
\end{equation*}
and so, applying~\eqref{thm:maximal_noncomp_ac_target:property_of_Y} to $f=Tx_j$ and $g=h_j-Tx_j$, we get, thanks to \eqref{thm:maximal_noncomp_ac_target:sequence_qnorm} and \eqref{thm:maximal_noncomp_ac_target:eq2}, that
\begin{equation}\label{thm:maximal_noncomp_ac_target:eq3}
\|h_j\|_Y > \varepsilon_{r,\lambda} \quad \text{for every $j\in\N$},
\end{equation}
in which $\varepsilon_{r,\lambda}>0$ is independent of $j$.
On the other hand, since
\begin{align*}
\spt h_j \subseteq \spt Tx_j \to \emptyset \quad &\text{as $j\to\infty$} \\
\intertext{and}
|h_j| \leq \sum_{k = 1}^m |g_k| \in Y \quad &\text{$\mu$-a.e.~in $R$ for every $j\in\N$},
\end{align*}
it follows from the absolute continuity of the $\|\cdot\|_Y$-quasinorm that (see \cite[Chapter~1, Proposition~3.6]{BS_book:88}) 
\begin{equation}\label{thm:maximal_noncomp_ac_target:eq4}
\lim_{j \to \infty} \|h_j\|_Y = 0.
\end{equation}
However, \eqref{thm:maximal_noncomp_ac_target:eq3} and \eqref{thm:maximal_noncomp_ac_target:eq4} contradict each other. It thus follows that $\beta\ge\lambda$, as desired.
\end{proof}

\begin{remarks}\ 
\begin{enumerate}[1.]
	\item In particular, the assumptions on $Y$ in the preceding theorem are satisfied when $Y$ is a Banach function space (in the sense of \cite{BS_book:88}) whose norm is absolutely continuous.
	\item Loosely speaking, a function from a quasi-K\"othe function space $Y$ has absolutely continuous quasinorm $\|\cdot\|_Y$ if and only if it can be used as a dominating function for which a suitable dominated convergence theorem for $\|\cdot\|_Y$ is valid. Such a characterization is well known, but it is usually stated with unnecessary assumptions. For example, it is proved in \cite[Chapter~1, Proposition~3.6]{BS_book:88}, which is referenced in the preceding proof, but $Y$ is assumed there to be a Banach function space. However, one can readily verify that the proof carries over verbatim for quasi-K\"othe function spaces.
    \item The assumption~\eqref{thm:maximal_noncomp_ac_target:sequence_spt} requires the supports of $\{Tx_j\}_{j = 1}^\infty$ to vanish pointwise $\mu$-a.e. Note that this is a more general assumption than requiring the measure of the supports to vanish. If the measure of the supports vanishes, then \eqref{thm:maximal_noncomp_ac_target:sequence_spt} is satisfied, but the opposite implication is valid only for finite measures in general.
    \item Even though we assume that $(R, \mu)$ is nonatomic throughout the paper, Theorem~\ref{thm:maximal_noncomp_ac_target} and its proof are valid for any $\sigma$-finite measure space.
    \item The assumption that $\|\cdot\|_Y$ is absolutely continuous is necessary, as the following simple example reveals. Consider $X=\ell_1$, $Y=\ell_{\infty}$, and let $T$ be the embedding operator $I\colon \ell_1 \to \ell_\infty$. Then all the other assumptions of Theorem~\ref{thm:maximal_noncomp_ac_target} are satisfied\textemdash in particular, the sequence $\{e^j\}_{j = 1}^\infty$ satisfies $\spt e^j\to\emptyset$ and $\|e^j\|_{Y}=\|T\|=1$, but $T$ is not maximally noncompact; in fact, we have $\beta(T)=\frac{1}{2}$ (cf.~\cite[Example, page~9408]{Lan:21}).
    \item The maximal noncompactness of the nonlimiting Sobolev embeddings \eqref{E:sobolev-lebesgue-nonlimiting} and \eqref{E:sobolev-lorentz-lorentz} can be easily obtained as a simple corollary of Theorem~\ref{thm:maximal_noncomp_ac_target}.
 \end{enumerate}
\end{remarks}

The following proposition can be used for verifying that the assumptions on $Y$ in Theorem~\ref{thm:maximal_noncomp_ac_target} for $Y = \Lambda^q_w(R,\mu)$ are satisfied.
\begin{proposition}\label{prop:Lambda_space_suff_condition_ac_theorem}
Let $q\in(0, \infty)$. Let $w\colon (0, \mu(R)) \to (0, \infty)$ be continuous, and assume that \eqref{prel:Lambda_space_nontrivial}, \eqref{prel:Lambda_space_deltaTwo}, and \eqref{prel:Lambda_space_quasiKothe} are satisfied. Furthermore, set
\begin{equation}\label{prop:Lambda_space_suff_condition_ac_theorem:dil_bdd}
\Theta(\lambda)  = \sup_{t\in(0, \lambda \mu(R))}\frac{w\big( \frac{t}{\lambda} \big)}{\lambda w(t)},\ \lambda\in(0, 1),
\end{equation}
and assume that $\Theta(\lambda)$ is finite for every $\lambda\in (0, 1)$ and  that
\begin{equation}\label{prop:Lambda_space_suff_condition_ac_theorem:control_func_small}
\inf_{\lambda\in(0,1)} \Theta(\lambda) \leq 1.
\end{equation}
Then the Lambda space $Y = \Lambda^q_w(R,\mu)$ satisfies the assumptions on $Y$ in Theorem~\ref{thm:maximal_noncomp_ac_target}\textemdash in other words, it is a quasi-K\"othe function space with absolutely continuous and uniformly separating quasinorm.
\end{proposition}
\begin{proof}
Recall that, since the assumptions \eqref{prel:Lambda_space_nontrivial}, \eqref{prel:Lambda_space_deltaTwo}, and \eqref{prel:Lambda_space_quasiKothe} on $w$ are satisfied, $\Lambda^q_w(R,\mu)$ is a quasi-K\"othe function space with absolutely continuous quasinorm.

Let $f,g\in \Lambda^q_w(R,\mu)$ be such that
\begin{equation}\label{prop:Lambda_space_suff_condition_ac_theorem:f_g_norms}
\|g\|_{\Lambda^q_w(R,\mu)} \geq R > r \geq \|f\|_{\Lambda^q_w(R,\mu)}.
\end{equation}
We need to show that
\begin{equation}\label{prop:Lambda_space_suff_condition_ac_theorem:desired}
\|f + g\|_{\Lambda^q_w(R,\mu)} \geq \varepsilon_{r,R}
\end{equation}
for some $\varepsilon_{r,R} > 0$ independent of the functions $f,g$. Clearly, we may assume that $\|f\|_{\Lambda^q_w(R,\mu)} > 0$. Let $\lambda \in (0, 1)$. Using \eqref{prel:star_of_sum_pointwise_est}, we have
\begin{align}
\|g\|_{\Lambda^q_w(R,\mu)} &\leq \|(f + g)^*(\lambda t) w(t)^\frac1{q} \|_{L^q(0, \mu(R))} \nonumber\\
&\quad + \|f^*((1-\lambda)t) w(t)^\frac1{q}\|_{L^q(0, \mu(R))}. \label{prop:Lambda_space_suff_condition_ac_theorem:eq1}
\end{align}

Next, using \eqref{prop:Lambda_space_suff_condition_ac_theorem:dil_bdd}, we have
\begin{align*}
\|(f + g)^*(\lambda \cdot) w(t)^\frac1{q} \|_{L^q(0, \mu(R))}^q &= \int_0^{\mu(R)} (f + g)^*(\lambda t)^q w(t) \dd t \\
&= \int_0^{\lambda \mu(R)} (f + g)^*(t)^q w\Big( \frac{t}{\lambda} \Big) \lambda^{-1}\dd t \\
&\leq \Theta(\lambda) \|f + g\|_{\Lambda^q_w(R,\mu)}^q.
\end{align*}
Replacing $f + g$ with $f$ and $\lambda$ with $1-\lambda$, we also obtain
\begin{equation*}
\|f^*((1-\lambda)\cdot) w(t)^\frac1{q}\|_{L^q(0, \mu(R))} \leq \Theta(1 - \lambda)^{\frac1{q}} \|f\|_{\Lambda^q_w(R,\mu)}.
\end{equation*}
Combining these two estimates with \eqref{prop:Lambda_space_suff_condition_ac_theorem:eq1}, we arrive at
\begin{equation*}
\|g\|_{\Lambda^q_w(R,\mu)} \leq \Theta(\lambda)^{\frac1{q}} \|f + g \|_{\Lambda^q_w(R,\mu)} + \Theta(1 - \lambda)^{\frac1{q}} \|f\|_{\Lambda^q_w(R,\mu)}.
\end{equation*}
Hence
\begin{equation}\label{prop:Lambda_space_suff_condition_ac_theorem:eq2}
\|f + g \|_{\Lambda^q_w(R,\mu)} \geq \Theta(\lambda)^{-\frac1{q}} \big( \|g\|_{\Lambda^q_w(R,\mu)} - \Theta(1 - \lambda)^{\frac1{q}} \|f\|_{\Lambda^q_w(R,\mu)} \big)
\end{equation}
for every $\lambda\in(0,1)$.

Now, since $r < R$, we can find $\lambda_0\in (0,1)$ such that
\begin{equation*}
R - \Theta(1 - \lambda_0)^{\frac1{q}} r > 0
\end{equation*}
thanks to \eqref{prop:Lambda_space_suff_condition_ac_theorem:control_func_small}. Finally, combining this with \eqref{prop:Lambda_space_suff_condition_ac_theorem:eq2} and \eqref{prop:Lambda_space_suff_condition_ac_theorem:f_g_norms}, we obtain
\begin{equation*}
\|f + g \|_{\Lambda^q_w(R,\mu)} \geq \Theta(\lambda_0)^{-\frac1{q}} \big( R - \Theta(1 - \lambda_0)^{\frac1{q}} r \big) > 0.
\end{equation*}
Therefore, \eqref{prop:Lambda_space_suff_condition_ac_theorem:desired} is valid with $\varepsilon_{r,R} = \Theta(\lambda_0)^{-\frac1{q}} \big( R - \Theta(1 - \lambda_0)^{\frac1{q}} r \big)$.
\end{proof}

\begin{remark}\label{rem:ac_theorem_Lambda_spaces_examples}
Simple\textemdash yet important\textemdash examples of Lambda spaces $\Lambda^q_w(R,\mu)$ to which Proposition~\ref{prop:Lambda_space_suff_condition_ac_theorem} applies are Lorentz spaces $L^{p,q}(R,\mu)$ with either $p=1$ and $q\in(0, 1]$ or $p\in(1, \infty)$ and $q\in(0, \infty)$.

Another example, which will be important in the next section, is provided by Lorentz--Zygmund spaces $L^{p,q,\alpha}(R,\mu)$, where $\mu(R) < \infty$, $q\in(0, \infty)$, and one of the conditions \eqref{prel:LZ_quasi_Kothe_cond} is satisfied. Recall that the weight $w$ is defined by \eqref{prel:LZ_weight_def}. To this end, when $\alpha\geq 0$, note that
\begin{equation*}
\Theta(\lambda)  = \lambda^{-\frac{q}{p}} \lim_{t \to 0^+} \frac{\Big(1 + \log\big( \frac{\lambda \mu(R)}{t}\big) \Big)^{\alpha q}}{1 + \log\big( \frac{\mu(R)}{t} \big)^{\alpha q}} = \lambda^{-\frac{q}{p}}
\end{equation*}
for every $\lambda\in(0,1)$. On the other hand, when $\alpha < 0$, we have
\begin{equation*}
\Theta(\lambda)  = \lambda^{-\frac{q}{p}} \lim_{t \to \lambda\mu(R)^-} \frac{\Big(1 + \log\big( \frac{\lambda \mu(R)}{t}\big) \Big)^{\alpha q}}{\Big(1 + \log\big( \frac{\mu(R)}{t}\big) \Big)^{\alpha q}} = \lambda^{-\frac{q}{p}} \frac{\log(2)^{\alpha q}}{\log\big( \frac{2}{\lambda}\big)^{\alpha q}}
\end{equation*}
for every $\lambda\in(0,1)$, and $\lim_{\lambda \to 1^-} \Theta(\lambda) = 1$.
\end{remark}

An important feature of Theorem~\ref{thm:maximal_noncomp_ac_target} is that it does not require $\|\cdot\|_Y$ to be disjointly superadditive nor a norm. Consequently, we will be able to use it to prove the maximal noncompactness of the limiting Sobolev embeddings of type~\eqref{E:sobolev-lorentz-zygmund}.

We conclude this section with a characterization of when the Lambda spaces $\Lambda_w^q(R, \mu)$ are disjointly superadditive. As a corollary, we will obtain that the Lorentz--Zygmund spaces with $p = \infty$ (in particular, those of Brezis--Wainger type, appearing in~\eqref{E:sobolev-lorentz-zygmund}) are never disjointly superadditive. Note that the following characterization was already obtained in \cite[Theorem~7]{KM:04} (see also references therein), but its proof is omitted there. We prove it here for the reader's convenience.
\begin{proposition}
Let $q\in(0, \infty)$. The Lambda space $\Lambda_w^q(R, \mu)$ with a weight $w$ satisfying \eqref{prel:Lambda_space_nontrivial} and \eqref{prel:Lambda_space_deltaTwo} is $\gamma$-disjointly superadditive if and only if $\gamma\geq q$ and the function $t\mapsto W(t)t^{-\frac{q}{\gamma}}$ is equivalent to a nondecreasing function on $(0, \mu(R))$.
\end{proposition}
\begin{proof}
Assume that $\Lambda_w^q(R, \mu)$ is $\gamma$-disjointly superadditive.

By \cite[Theorem~1]{KM:04}, there are constants $C_1, C_2>0$ and functions $\{f_j\}_{j = 1}^\infty$ such that their supports are mutually disjoint, $\|f_j\|_{\Lambda_w^q(R, \mu)} = 1$ and $C_1^q \sum_{j = 1}^\infty |\alpha_j|^q\leq \|\sum_{j = 1}^\infty \alpha_j f_j\|_{\Lambda_w^q(R, \mu)}^q \leq C_2^q \sum_{j = 1}^\infty |\alpha_j|^q$ for every $\{\alpha_j\}_{j = 1}^\infty\subseteq\ell_q$. This combined with the $\gamma$-disjoint superadditivity of $\Lambda_w^q(R, \mu)$ implies that
\begin{equation*}
\sum_{j = 1}^\infty |\alpha_j|^\gamma \leq C C_2^\gamma \left( \sum_{j = 1}^\infty |\alpha_j|^q \right)^\frac{\gamma}{q}
\end{equation*}
for every $\{\alpha_j\}_{j = 1}^\infty\subseteq\ell_q$. In other words, $\ell_q \subseteq \ell_\gamma$; hence $\gamma \geq q$.

Next, we show that the function $t\mapsto W(t)t^{-\frac{q}{\gamma}}$ is equivalent to a nondecreasing function\textemdash namely
\begin{equation*}
F(t) = \sup_{s\in (0, t]} W(s)s^{-\frac{q}{\gamma}},\ t\in(0, \mu(R)).
\end{equation*}
On the one hand, we clearly have $W(t)t^{-\frac{q}{\gamma}} \leq F(t)$ for every $t\in(0, \mu(R))$. On the other hand, we claim that
\begin{equation}\label{prop:lambda_disj_superadd_charac:eq1}
F(t) \leq (2C)^\frac{q}{\gamma} W(t)t^{-\frac{q}{\gamma}} \quad \text{for every $t\in(0, \mu(R))$}.
\end{equation}
To this end, fix $0 < s \leq t < \mu(R)$ and set $k = \lfloor \frac{t}{s} \rfloor$. Note that $k\leq \frac{t}{s} \leq 2k$. Since $(R,\mu)$ is nonatomic, there are disjoint sets $\{E_j\}_{j = 1}^k\subseteq R$ such that $\mu(E_j) = \frac{t}{k}$ for every $j = 1, \dots, k$. Combining this with the $\gamma$-disjoint superadditivity of $\Lambda_w^q(R, \mu)$, we obtain
\begin{align*}
W(t)^\frac{\gamma}{q} &= \|\chi_{(0,t)}w\|_{L^1(0, \mu(R))}^\frac{\gamma}{q} = \Big\| \sum_{j = 1}^k \chi_{E_j} \Big\|_{\Lambda_w^q(R, \mu)}^\gamma \geq \frac1{C} \sum_{j = 1}^k \| \chi_{E_j} \|_{\Lambda_w^q(R, \mu)}^\gamma \\
&= \frac{k}{C} W\Big( \frac{t}{k} \Big)^\frac{\gamma}{q} \geq \frac{t}{2C}\frac{W(s)^\frac{\gamma}{q}}{s},
\end{align*}
whence \eqref{prop:lambda_disj_superadd_charac:eq1} follows.

Now, assume $\gamma\geq q$ and that there is a nondecreasing function $F(t)$ such that $K_1 F(t) \leq W(t)t^{-\frac{q}{\gamma}} \leq K_2 F(t)$ for some $K_1,K_2>0$ and every $t\in(0, \mu(R))$. Set
\begin{equation*}
V(t) = \int_0^t F(s)^\frac{\gamma}{q} \dd s,\ t\in[0, \mu(R)).
\end{equation*}
Since the function $F^\frac{\gamma}{q}$ is nondecreasing, $V$ is convex. We clearly also have $V(0) = 0$. It follows that $V$ is superadditive. 

Next, we claim that
\begin{equation}\label{prop:lambda_disj_superadd_charac:eq2}
K_1^{\frac{\gamma}{q}} V(t) \leq W(t)^{\frac{\gamma}{q}} \leq 2K_3K_2^{\frac{\gamma}{q}} V(t) \quad \text{for every $t\in(0, \mu(R))$},
\end{equation}
where $K_3$ is the supremum in \eqref{prel:Lambda_space_deltaTwo}. To this end, on the one hand, we have
\begin{equation*}
V(t) \leq tF(t)^\frac{\gamma}{q} \leq K_1^{-\frac{\gamma}{q}} W(t)^{\frac{\gamma}{q}}
\end{equation*}
thanks to the monotonicity of $F^\frac{\gamma}{q}$. On the other hand, using \eqref{prel:Lambda_space_deltaTwo}, we obtain
\begin{equation*}
V(t) \geq \frac1{K_2^{\frac{\gamma}{q}}} \int_{\frac{t}{2}}^t \frac{W(s)^{\frac{\gamma}{q}}}{s} \dd s \geq \frac{W(\frac{t}{2})^{\frac{\gamma}{q}}}{2K_2^{\frac{\gamma}{q}}} \geq \frac{W(t)^{\frac{\gamma}{q}}}{2K_3K_2^{\frac{\gamma}{q}}}.
\end{equation*}

Now, let $k\in\N$ and $\{f_j\}_{j=1}^{k}\subseteq \Lambda_w^q(R, \mu)$ be functions having pairwise disjoint supports. Since their supports are disjoint, we have $(\sum_{j = 1}^k f_j)_* = \sum_{j = 1}^k (f_j)_*$. Combining this, \eqref{prel:lambda_by_means_of_distributional_func}, \eqref{prop:lambda_disj_superadd_charac:eq2}, and the superadditivity of $V$, we obtain
\begin{align*}
\Big\| \sum_{j  = 1}^k f_j \Big\|_{\Lambda_w^q(R, \mu)}^q &= q \int_0^{\mu(R)} W\Big( \sum_{j = 1}^k (f_j)_*(t) \Big) t^{q - 1} \dd t \\
&= q \int_0^{\mu(R)} \Bigg( W\Big( \sum_{j = 1}^k (f_j)_*(t) \Big)^\frac{\gamma}{q} \Bigg)^\frac{q}{\gamma} t^{q - 1} \dd t \\
&\geq qK_1 \int_0^{\mu(R)} V\Big( \sum_{j = 1}^k (f_j)_*(t) \Big)^\frac{q}{\gamma} t^{q - 1} \dd t \\
&\geq qK_1 \int_0^{\mu(R)} \Big( \sum_{j = 1}^k V((f_j)_*(t)) \Big)^\frac{q}{\gamma} t^{q - 1} \dd t \\
&= qK_1 \Big \|\sum_{j = 1}^k V((f_j)_*(t)) t^\frac{\gamma(q - 1)}{q} \Big\|_{L^\frac{q}{\gamma}(0, \mu(R))}^\frac{q}{\gamma},
\end{align*}
whence
\begin{equation}\label{prop:lambda_disj_superadd_charac:eq3}
\Big\| \sum_{j  = 1}^k f_j \Big\|_{\Lambda_w^q(R, \mu)}^\gamma \geq (qK_1)^\frac{\gamma}{q} \Big \|\sum_{j = 1}^k V((f_j)_*(t)) t^\frac{\gamma(q - 1)}{q} \Big\|_{L^\frac{q}{\gamma}(0, \mu(R))}.
\end{equation}
Since $\frac{q}{\gamma}\in(0, 1]$ and the functions $t\mapsto V((f_j)_*(t)) t^\frac{\gamma(q - 1)}{q}$, $j = 1, \dots, k$, are nonnegative, we can use the reverse triangle inequality for the $\|\cdot\|_{L^\frac{q}{\gamma}(0, \mu(R))}$ quasinorm to obtain
\begin{equation}\label{prop:lambda_disj_superadd_charac:eq4}
\Big \|\sum_{j = 1}^k V((f_j)_*(t)) t^\frac{\gamma(q - 1)}{q} \Big\|_{L^\frac{q}{\gamma}(0, \mu(R))} \geq \sum_{j = 1}^k \| V((f_j)_*(t)) t^\frac{\gamma(q - 1)}{q} \|_{L^\frac{q}{\gamma}(0, \mu(R))}.
\end{equation}

Finally, combining \eqref{prop:lambda_disj_superadd_charac:eq3} and \eqref{prop:lambda_disj_superadd_charac:eq4} together with using \eqref{prel:lambda_by_means_of_distributional_func} and \eqref{prop:lambda_disj_superadd_charac:eq2} again, we arrive at
\begin{align*}
\Big\| \sum_{j  = 1}^k f_j \Big\|_{\Lambda_w^q(R, \mu)}^\gamma &\geq (qK_1)^\frac{\gamma}{q} \sum_{j = 1}^k \| V((f_j)_*(t)) t^\frac{\gamma(q - 1)}{q} \|_{L^\frac{q}{\gamma}(0, \mu(R))} \\
&\geq  \frac{(K_1)^\frac{\gamma}{q}}{2 K_3 K_2^\frac{\gamma}{q}} q^\frac{\gamma}{q} \sum_{j = 1}^k \| W((f_j)_*(t))^\frac{\gamma}{q} t^\frac{\gamma(q - 1)}{q} \|_{L^\frac{q}{\gamma}(0, \mu(R))} \\
&=  \frac{(K_1)^\frac{\gamma}{q}}{2 K_3 K_2^\frac{\gamma}{q}}  \sum_{j = 1}^k \| f_j \|_{\Lambda_w^q(R, \mu)}^\gamma. \qedhere
\end{align*}
\end{proof}

\begin{corollary}
Let $\gamma > 0$ and $q\in(0, \infty)$. Let either $p\in(0, \infty)$ and $\alpha \in \R$ or $p = \infty$ and $\alpha + \frac1{q} < 0$. The Lorentz--Zygmund space $L^{p,q,\alpha}(R, \mu)$ is $\gamma$-disjointly superadditive if and only if
\begin{itemize}
    \item $0<p<q\leq \gamma$ or
    \item $0<q\leq p < \gamma$ or
    \item $0 < q \leq p = \gamma$ and $\alpha\leq 0$.
\end{itemize}
\end{corollary}


\section{Maximal noncompactness of limiting Sobolev embeddings}
In this section, we will finally prove that the limiting Sobolev embeddings~\eqref{E:sobolev-lorentz-zygmund} and~\eqref{E:sobolev-lorentz-zygmund_weak} are maximally noncompact. 
\begin{theorem}
    Let $\Omega\subseteq \rn$ be a domain with $|\Omega| < \infty$. Assume that $n\in\N$, $n\ge2$, and $q\in[n,\infty]$. Set $\alpha = -1 + \frac1{n} - \frac1{q}$.  Then the Sobolev embedding
    \begin{equation}\label{thm:maximal-noncompact-Sobolev-into-BW:em}
        V_0^{1, n}(\Omega)\hookrightarrow L^{\infty,q, \alpha}(\Omega)
    \end{equation}
    is maximally noncompact.
\end{theorem}
\begin{proof}
We may assume without loss of generality that $0\in\Omega$. 

Fix any $\lambda \in (0, \|I\|)$, where $\|I\|$ is the norm of the embedding \eqref{thm:maximal-noncompact-Sobolev-into-BW:em}. Since $\lambda < \|I\|$, we can find a function $u\in B_{V_0^{1, n}(\Omega)}$ such that $\|u\|_{L^{\infty,q, \alpha}(\Omega)} > \lambda$. Let $B_R\subseteq \rn$ be the ball centered at the origin with radius $R$ such that $|B_R| = |\Omega|$. It follows from the P\'olya-Szeg\"o inequality (e.g., \cite[Chapter~15]{L_book:17}, see also~\cite[Lemma~4.1]{CP:98}) that $u^\bigstar\in B_{V_0^{1,n}(B_R)}$, where $u^\bigstar$ is the spherically symmetric rearrangement of $u$. Recall that $u^\bigstar$ is defined as
\begin{equation*}
u^\bigstar(x) = u^*(\omega_n |x|^n),\ x\in\rn.
\end{equation*}
Moreover, since the functions $u$ and $u^\bigstar$ are equimeasurable (i.e., their distributional functions are the same) and $|B_R| = |\Omega|$, we have
\begin{equation*}
\|u^\bigstar\|_{L^{\infty,q, \alpha}(B_R)} = \|u\|_{L^{\infty,q, \alpha}(\Omega)} > \lambda.
\end{equation*}
Furthermore, by using a suitable sequence of cutoff functions, there is a sequence of radially symmetric functions $\{v_j\}_{j = 1}^\infty\subseteq B_{V_0^{1,n}(B_R)}$ such that $\|v_j\|_{L^{\infty,q, \alpha}(B_R)}\nearrow \|u^\bigstar\|_{L^{\infty,q, \alpha}(B_R)}$ and $\spt v_j \nearrow \spt u^\bigstar$ as $j\to\infty$. Hence, there is a function
\begin{equation}\label{thm:maximal_noncomp_ac_target:eq6}
v\in B_{V_0^{1,n}(B_R)}
\end{equation}
such that $\spt v \subseteq B_{\tilde{R}}$ for some $\tilde{R} \in (0, R)$, $v = v^\bigstar$, and
\begin{equation}\label{thm:maximal_noncomp_ac_target:eq10}
\|v\|_{L^{\infty,q, \alpha}(B_R)} > \lambda.
\end{equation}

Now, for every $\kappa\in(0,1)$, we define the function $v_\kappa \in V_0^{1,n}(B_R)$ (cf.~\cite{I:19}) as
\begin{align*}
v_\kappa(x) &= \kappa^{-1 + \frac1{n}} v\left( \frac{|x|^{\kappa - 1}}{(2^\frac1{n} R)^{\kappa - 1}} x \right) \\
	&=  \kappa^{-1 + \frac1{n}} v^* \left( \omega_n \left( \frac{|x|^{\kappa}}{(2^\frac1{n} R)^{\kappa - 1}} \right)^n \right),\ x\in \rn.
\end{align*}
Note that $v_\kappa$ is supported inside $B_{R_\kappa}$, where
\begin{equation*}
R_\kappa = 2^\frac{\kappa - 1}{n\kappa} \left( \frac{\tilde{R}}{R} \right)^\frac1{\kappa} R.
\end{equation*}
Since $R_\kappa \to 0$ as $\kappa \to 0^+$, we have
\begin{equation}\label{thm:maximal_noncomp_ac_target:eq12}
\spt v_\kappa \to \emptyset \quad \text{as $\kappa \to 0^+$}.
\end{equation}

Let $\kappa_0\in(0,1)$ be such that $\spt v_\kappa \subsetneq \Omega \cap B_R$ for every $\kappa\in(0, \kappa_0)$. Fix arbitrary $\kappa\in(0, \kappa_0)$. Since $\spt v_\kappa \subsetneq \Omega$, we have $v_\kappa \in V_0^{1,n}(\Omega)$. Before we can conclude the proof, we need to observe three things.

First, we have
\begin{equation}\label{thm:maximal_noncomp_ac_target:eq5}
\| \nabla v_\kappa\|_{L^n(B_R)} = \| \nabla v\|_{L^n(B_R)}.
\end{equation}
To this end, it can be easily verified that
\begin{equation*}
|\nabla v (x)| = \phi(|x|) \quad \text{for a.e.~$x\in\rn$},
\end{equation*}
where
\begin{equation*}
\phi(t) = -(v^*)'(\omega_n t^n) n\omega_n t^{n - 1},
\end{equation*}
and that
\begin{equation*}
|\nabla v_\kappa (x)| = \kappa^\frac1{n} \phi\Bigg( \frac{|x|^\kappa}{(2^\frac1{n} R)^{\kappa - 1}} \Bigg) \frac{|x|^{\kappa - 1}}{(2^\frac1{n} R)^{\kappa - 1}} \quad \text{for a.e.~$x\in\rn$}.
\end{equation*}
Hence
\begin{align*}
	\| \nabla v_\kappa \|_{L^n(B_R)}^n &= n \omega_n \kappa \int_0^R \phi\Bigg( \frac{r^\kappa}{(2^\frac1{n} R)^{\kappa - 1}} \Bigg)^n \frac{r^{n(\kappa - 1)}}{(2^\frac1{n} R)^{n(\kappa - 1)}} r^{n-1}\dd r \\
	&= n \omega_n \kappa \int_0^R \left( \phi\Bigg( \frac{r^\kappa}{(2^\frac1{n} R)^{\kappa - 1}} \Bigg) \frac{r^\kappa}{(2^\frac1{n} R)^{\kappa - 1}} \right)^n \ddov{r} \\
		&= n \omega_n \kappa \int_0^{2^{\frac{1 - \kappa}{n}}R} \phi(s)^n s^n \kappa^{-1} s^{-\frac1{\kappa}} s^{\frac1{\kappa} - 1} \dd s \\
		&= \omega_n \kappa \int_0^{R} \phi(s)^n s^{n-1} \dd s \\
		&= \| \nabla v \|_{L^n(B_R)}^n.
\end{align*}
Here we used the fact that $2^{\frac{1 - \kappa}{n}} > 1$ in the next to last equality. Therefore, since $\spt v_\kappa \subseteq \Omega$, we arrive at
\begin{equation}\label{thm:maximal_noncomp_ac_target:eq7}
v_\kappa \in B_{V_0^{1,n}(\Omega)}
\end{equation}
by combining \eqref{thm:maximal_noncomp_ac_target:eq6} and \eqref{thm:maximal_noncomp_ac_target:eq5}.

Second, note that $v_\kappa(x) = g(\omega_n |x|^n)$, where $g\colon (0, \infty) \to [0, \infty)$ is defined as
\begin{equation}\label{thm:maximal_noncomp_ac_target:eq8}
g(t) = \kappa^{-1 + \frac1{n}} v^*\left( (2|B_R|)^{1 - \kappa} t^\kappa \right),\ t\in(0, \infty),
\end{equation}
is nonincreasing.

Third, we have
\begin{equation}\label{thm:maximal_noncomp_ac_target:eq9}
\|v_\kappa\|_{L^{\infty,q, \alpha}(B_R)} = \| v \|_{L^{\infty,q, \alpha}(B_R)}.
\end{equation}
Indeed, if $q\in[n, \infty)$, using \eqref{thm:maximal_noncomp_ac_target:eq8}, we obtain
\begin{align*}
\|v_\kappa\|_{L^{\infty,q, \alpha}(B_R)}^q &= \kappa^{q(-1 + \frac1{n})} \int_0^{|B_R|}  v^*\left( (2 |B_R|)^{1 - \kappa} t^\kappa \right)^q \log\left( \frac{2|B_R|}{t} \right)^{\alpha q} \ddov{t} \\
 &= \kappa^{q(-1 + \frac1{n}) - 1} \int_0^{2^{1 - \kappa}|B_R|}  v^*(s)^q \log\Bigg( \left( \frac{2|B_R|}{s} \right)^\frac1{\kappa} \Bigg)^{\alpha q} \ddov{s} \\
 &= \kappa^{q(-1 + \frac1{n}) - 1 - \alpha q} \int_0^{2^{1 - \kappa}|B_R|}  v^*\left( s \right)^q \log\left( \frac{2|B_R|}{s}  \right)^{\alpha q} \ddov{s} \\
 &= \int_0^{|B_R|}  v^*\left( s \right)^q \log\left( \frac{2|B_R|}{s}  \right)^{\alpha q} \ddov{s} \\
 &= \|v\|_{L^{\infty,q, \alpha}(B_R)}^q.
\end{align*}
Here we also used the fact that $2^{1-\kappa} > 1$ in the next to last equality. If $q = \infty$, we proceed similarly. We have
\begin{align*}
\|v_\kappa\|_{L^{\infty,\infty, \alpha}(B_R)} &= \kappa^{-1 + \frac1{n}}  \sup_{t\in(0, |B_R|)} v^*\left( (2 |B_R|)^{1 - \kappa} t^\kappa \right) \log\left( \frac{2|B_R|}{t} \right)^{-1 + \frac1{n}} \\
&= \kappa^{-1 + \frac1{n}}  \sup_{t\in(0, 2^{1 - \kappa}|B_R|)} v^*(t) \log\left( \left( \frac{2|B_R|}{t} \right)^\frac1{\kappa} \right)^{-1 + \frac1{n}} \\
&= \|v\|_{L^{\infty,\infty, \alpha}(B_R)}.
\end{align*}
Therefore, since $\spt v_\kappa \subseteq \Omega \cap B_R$ and $|B_R| = |\Omega|$, we obtain
\begin{equation}\label{thm:maximal_noncomp_ac_target:eq11}
\|v_\kappa\|_{L^{\infty,q, \alpha}(\Omega)} > \lambda
\end{equation}
by combining \eqref{thm:maximal_noncomp_ac_target:eq10} and \eqref{thm:maximal_noncomp_ac_target:eq9}.

Finally, in view of \eqref{thm:maximal_noncomp_ac_target:eq12}, \eqref{thm:maximal_noncomp_ac_target:eq7}, and \eqref{thm:maximal_noncomp_ac_target:eq11}, the maximal noncompactness of the embedding \eqref{thm:maximal-noncompact-Sobolev-into-BW:em} follows from either Theorem~\ref{thm:maximal_noncomp_ac_target} together with Remark~\ref{rem:ac_theorem_Lambda_spaces_examples} if $q\in[n, \infty)$ or \cite[Corollary~3.3]{MMMP:preprint} if $q = \infty$.
\end{proof}

\subsubsection*{Acknowledgment}
We would like to thank the referee for the careful reading of the paper and their valuable comments.

\end{document}